\documentclass[10pt]{article}

\usepackage[authoryear,round]{natbib}                   
\usepackage{amsmath}                                                        
\usepackage{graphicx}                                                       
\usepackage{subfigure}
\usepackage{enumitem}                                                    
\usepackage{hyperref}
\usepackage{xcolor}
\usepackage{amssymb}                                                        
\usepackage[mathscr]{eucal}                                             
\usepackage{cancel}                                                             
\usepackage[normalem]{ulem}                                                                 
\usepackage{pstricks}
\usepackage{rotating}
\usepackage{lscape}
\usepackage[paperwidth=8.5in,paperheight=11in,top=1.25in, bottom=1.25in, left=1.00in, right=1.00in]{geometry}
\usepackage{mathtools}                                                      
\mathtoolsset{showonlyrefs=true}                                    
\usepackage{fixltx2e,amsmath}                                           
\MakeRobust{\eqref}

\usepackage{mathdots}
\usepackage{amsthm}                                                             
\allowdisplaybreaks                                                             
\usepackage{tikz}
\usetikzlibrary{decorations.pathmorphing}
\usepackage{pgfplots}
\pgfplotsset{compat=1.16}

\theoremstyle{plain}
\newtheorem{theorem}{Theorem}

\newtheorem{lemma}[theorem]{Lemma}                              
\newtheorem{proposition}[theorem]{Proposition}

\theoremstyle{definition}
\newtheorem{definition}[theorem]{Definition}
\newtheorem{example}[theorem]{Example}
\newtheorem{remark}[theorem]{Remark}

\def \y {{\eta}}

\def \a {{\alpha}}
\def \b {{\beta}}
\def \d {{\delta}}

\newcommand\N{\mathbb{N}}

\newcommand\R{\mathbb{R}}

\newcommand\Lc{\mathscr{L}}

\newcommand\eps{\varepsilon}

\def \ddY {{\mathbf{\vartheta}}}

\def \R  {{\mathbb {R}}}
\def \x {{\xi}}

\def \eps {{\varepsilon}}
\def \t {{\tau}}

\def \y {{\eta}}

\def \z {{\zeta}}

\def \a {{\alpha}}
\def \O {{\Omega}}

\def \d {{\delta}}

\def \a {{\alpha}}
\def \b {{\beta}}
\def \d {{\delta}}

\def \R {{\mathbb {R}}}
\def \N {{\mathbb {N}}}

\def \x {{\xi}}

\def \eps {{\varepsilon}}

\def \t {{\tau}}

\def \y {{\eta}}

\def \z {{\zeta}}

\def \O {{\Omega}}
\def \phi {{\varphi}}

\def \tilde {\widetilde}

\def\l {\lambda}

\def \à {{\`a }}
\def \è {{\`e }}
\def \ò {{\`o }}
\def \ù {{\`u }}
\def \I {\mathcal{I}}

\newcommand{\norm}[1]{\left\|{#1}\right\|}

\newcommand\dd{\mathrm{d}}

\newcommand\Rd{\mathbb{R}^{d}}
\newcommand\Rdd{\mathbb{R}\times\mathbb{R}^{2d}}

\newcommand\Ndzero{\mathbb{N}^d_{0}}

\newrgbcolor{blue}{.2 0 .7}
\newrgbcolor{lgrey}{.9 .9 .9}
\newrgbcolor{grey}{.7 .7 .7}
\newrgbcolor{dgrey}{.4 .4 .4}
\newrgbcolor{lgreen}{.8 1 .8}
\newrgbcolor{lred}{1 .8 .8}
\newrgbcolor{mlred}{1 .7 .7}
\newrgbcolor{bblue}{.6 .6 1}
\newrgbcolor{lblue}{.8 .8 1}
\newrgbcolor{llblue}{.93 .93 1}
\newrgbcolor{mblue}{.2 .25 1}
\newrgbcolor{dblue}{.2 .25 1}
\newrgbcolor{ddblue}{.4 .4 1}
\newrgbcolor{lgreen}{.8 1 .8}
\newrgbcolor{llgreen}{.85 1 .85}
\definecolor{lllgreen}{rgb}{.92 1 .92}
\newrgbcolor{mgreen}{ .7 1 .7}
\newrgbcolor{ggreen}{ .1 .75 .35}
\newrgbcolor{forestgreen}{0.13, 0.55, 0.13}
\newrgbcolor{brickred}{ .8, .25, .33}

\newcommand{\ZZ}{Z}
\newcommand{\XX}{X}

\def \phi {{\varphi}}
\def \eps {{\varepsilon}}

\begin{document}

\title{Intrinsic H\"older spaces for fractional kinetic operators
}
\author{
Maria Manfredini
\thanks{Dipartimento di Scienze Fisiche, Informatiche e Matematiche Universit\`a di Modena e Reggio Emilia, Modena, Italy.  \textbf{e-mail}: maria.manfredini@unimore.it}
\and
Stefano Pagliarani
\thanks{Dipartimento di Matematica, Universit\'a di Bologna, Bologna, Italy. \textbf{e-mail}: stefano.pagliarani9@unibo.it}
\and
Sergio Polidoro
\thanks{Dipartimento di Scienze Fisiche, Informatiche e Matematiche Universit\`a di Modena e Reggio Emilia, Modena, Italy.
\textbf{e-mail}: sergio.polidoro@unimore.it}
}

\date{This version: \today}

\maketitle

\begin{abstract}
We introduce anisotropic H\"older spaces useful for the study of the regularity theory for non local kinetic operators $\Lc$ whose prototypical example is
\begin{equation*}
  \Lc u (t,x,v) = \int_{\R^d}  \frac{C_{d,s}}{|v - v'|^{d+2s}} (u(t,x,v') - u(t,x,v))  \dd v' +  
  \langle   v , \nabla_x   \rangle + \partial_t, \quad (t,x,v)\in\R\times\R^{2d}.
\end{equation*}
The H\"older spaces are defined in terms of an anisotropic distance relevant to the Galilean geometric structure on $\R\times\R^{2d}$ the operator $\Lc$ is invariant with respect to. We prove an intrinsic Taylor-like formula, whose reminder is estimated in terms of the anisotropic distance of the Galilean structure. Our achievements naturally extend analogous known results for purely differential operators on Lie groups.
\end{abstract}


%
%

\section{Introduction}\label{intro}
We consider H\"older spaces and Taylor-like formulas useful for the study of the regularity theory of the solutions to $\Lc u = f$, being $\Lc$ a non local kinetic operator of the form
\begin{equation}\label{e1}
  \Lc u (t,x,v) = \int_{\R^d}  K(t,x,v,v') (u(t,x,v') - u(t,x,v))  \dd v' +  Yu,\qquad
  (t,x,v)\in\R\times\R^{2d}.
\end{equation}
where
\begin{equation}\label{eq:Y}
Y = \langle   v , \nabla_x   \rangle + \partial_t , \qquad v\in\R^d.
\end{equation}
The integral part is of order $2s$, with $s\in ]0,1[$, in the following sense
\begin{equation*}
    \frac{c^-}{|v - v'|^{d+2s}} \le K(t, x, v, v') \le \frac{c^+}{|v - v'|^{d+2s}}, \qquad (t,x,v,v')\in\R\times\R^{3d},
\end{equation*}
for some positive constants $c^-$ and $c^+$. One notable particular instance of $\Lc$ is the fractional kinetic Fokker–Planck operator
\begin{equation}\label{e1_prot}
\Lc_s = (-\Delta_v)^s + Y, 
\end{equation}
which is the operator in \eqref{e1} with
\begin{equation}\label{eq:kernel_prot}
    K(t, x, v, v') = \frac{C_{d,s}}{|v - v'|^{d+2s}},
\end{equation}
for a suitable positive constant $C_{d,s}$. Note that $\Lc_s = (-\Delta_v)^s + Y$ is also related to the infinitesimal generator of an $\alpha$-stable L\'evy process, with $\alpha = 2s$.

In the purely diffusive setting, which can be seen as the limiting case $s = 1$, $\Lc_1 = -\Delta_v + Y$ is a differential hypoelliptic operator. This means that every distributional solution $u$ to the equation $-\Delta_v u + Yu = f$ is a smooth function whenever $f$ is smooth. Indeed, setting the primary vector fields
\begin{equation}\label{eq:Z_fields}
\ZZ_i: = \partial_{v_i}, \qquad i=1,\cdots, d,
\end{equation} 
we obtain that $\Lc_1$ writes in the form $\Lc_1 = - \sum_{i=1}^d \ZZ_i^2 + Y$ and the system $\ZZ_1,\cdots,\ZZ_d,Y$ satisfies the so-called H\"ormander condition, namely 
\begin{equation}\label{eq:hormander}
\text{rank } \text{Lie}(\ZZ_1,\dots,\ZZ_d,Y) = 2d + 1.
\end{equation}
This is a straightforward consequence of 
\begin{equation}\label{eq:commutator}
[\partial_{v_j} , Y] = \partial_{x_j} , \qquad j=1,\dots, d.
\end{equation}
We emphasize that the regularity properties of the H\"ormander's operators depend on a non-Euclidean underlying structure (see the survey \cite{anceschi2019survey}). In the setting of the kinetic operator $\Lc$, this structure agrees with the Galilean translation (see \eqref{eq:translation} below).

In this work we rely on the geometric structure introduced for the {differential} operator $\Lc_1$, in order to study the fractional operator $\Lc_s$ for $0 < s < 1$. We give a definition of \emph{intrinsic H\"older spaces}, which extends the one indroduced by \cite{PPP16}, and we prove a Taylor polynomial approximation of a function $f$ belonging to this  H\"older space. 

\medskip

We conclude this section with some remarks about the intrinsic H\"older spaces and the applications of our main results. These spaces are \emph{anisotropic}, as the variables $v$ and $x$ in the Galileian group have a different role. Moreover, the definition of \emph{intrinsic H\"older spaces} is based on the non-Euclidean quasi-distance \eqref{eq:quasi_distance} of the underlying Galileian group. Remark \ref{rcaa} contains a brief discussion on our definition of anisotropic H\"older spaces compared with other definition present in the literature. 
In the purely differential setting, such H\"older spaces were studied by several authors, and a characterization relating the regularity along the vector fields $\partial_{v_1},\dots,\partial_{v_d},Y$ to the existence of appropriate instrinsic Taylor formulas was given by \cite{PPP16}. {We refer to the articles \cite{Bonfiglioli2009} and \cite{Arena} for similar Taylor formulas on homogeneous groups. The Taylor approximation of a solution to a PDE is a useful tool in the proof of Schauder estimates. We refer to \cite{PolidoroRebucciStroffolini2022},  where the regularity of classical solutions to degenerate Kolmogorov equations is obtained by using the method introduced in \cite{Wang2006} for uniformly elliptic and parabolic equations.} We also recall the article \cite{imbert2021schauder} where Taylor approximation results and Schauder estimates for kinetic equations are proved.
In the fractional setting, Schauder estimates for the solutions to $\Lc u = 0$ have been recently proved by \cite{imbert2018schauder}, in suitable H\"older spaces that take into account the intrinsic geometry of the Galilean group. We emphasize that the main results of this note do apply to the operators $\Lc$ considered in the aforementioned reference. 

Our main results also apply to a non-linear non-local version of \eqref{e1_prot}, that is
\begin{equation}\label{e1_prot-p}
\Lc_{s,p} = (-\Delta_v)^s_p + Y,  
\end{equation}
considered for $p \in ]1,\infty[$ and $s\in ]0,1[$ in \cite{AnceschiPiccinini2023}. In this case, the kernel $K$ in \eqref{e1} also depends on the unknown function $u$ and the term $(-\Delta_v)_p^s u$ takes the following form

\begin{equation}\label{eq:kernel_prot-p}
    (-\Delta_v)^s_p u(t, x, v) = \int_{\R^d}
    \frac{C_{d,p,s}}{|v - v'|^{d+ps}}|u(t,x,v) - u(t,x,v')|^{p-2}\big( u(t,x,v) - u(t,x,v')\big) \dd v'.
\end{equation}

\medskip

We finally remark that not only anisotropic spaces of H\"older continuous functions have been considered in literature for the study of kinetic equations of the form $\Lc u = f$. Indeed,  \cite{pascucci2022sobolev} prove intrinsic Taylor expansion for anisotropic Sobolev-Slobodeckij spaces, and prove continuous embeddings into Lorentz and intrinsic H\"older spaces. 

\medskip

{This article is organized as follows. Section 2 contains some recalls on the non-Euclidean geometry relevant to the kinetic operator $\Lc_1$. Note that the study of the non local operator $\Lc_s$ relies on the same geometric structure. In Section 2 the notion of intrinsic regularity and intrinsic H\"older spaces and the statement of Theorem \ref{th:main}, which is the main results of this note, are given. Some examples illustrate the meaning of the definitions and the main results. Section 3 contains some preliminaries the the proof of the main results, which is given in Section 4. Section 5 contains a local version of our main results, which generalizes Theorem \ref{th:main} in that it applies to the more general geometric framework of the \emph{non-homogeneous Lie groups}.}

\medskip

\noindent
{\bf Acknowldedgment.} We thank Luis Silvestre for advising us to investigate this subject.


\section{H\"older spaces and Taylor polynomials}
In this part we fix some notation for the geometric structure on $\R^{2d+1}$ that will be used in this work. We recall that, remarkably, $\Lc_s$ in \eqref{e1_prot} has the property of being invariant with respect to left translations in the group
$\left(\R\times\R^{2d},\circ\right)$, where the non-commutative group law ``$\circ$'' is defined by
\begin{equation}\label{eq:translation}
 z_1\circ z_2 = \left(t_1+t_2,  x_1 + x_2 + t_2 v_1 , v_1 + v_2  \right),\qquad z_1=(t_1,x_1,v_1),z_2=(t_2,x_2,v_2)\in \Rdd.
\end{equation}
Precisely, we have
\begin{equation}\label{eq:translation_invariance}
\big(\Lc u^{(z_1)}\big) (z_2)=(\Lc u)(z_1\circ z_2).
\end{equation}
where
\begin{equation}
 u^{(z_1)} (z_2)=u(z_1\circ z_2).
\end{equation}
This translation property is often referred to as ``Galilean'' change of coordinate and is very useful in kinetic theory. A systematic study of the PDEs theory on this group started in \cite{LanconelliPolidoro1994} in the limiting case $s=1$. Notice that $(\R\times\R^{2d},\circ)$ is a group with the identity and the inverse elements 
\begin{equation}
{\text{Id}=(0,0,0)}, \qquad (t,x,v)^{-1}=\left(-t,  t v - x  , -v \right).
\end{equation}
Moreover, $\Lc_s$ {\it is homogeneous of degree $\ddY= 2s$} with respect to the
dilations $\left(D_\lambda\right)_{\l>0}$ on $\R\times\R^{2d}$ given by
\begin{equation}\label{eq:dilation}
 D_\lambda =\textrm{diag}\big(\lambda^{\ddY},\lambda^{\ddY+1}I_{d},\lambda I_{d}\big),
\end{equation}
where $I_{d}$ is the $(d\times d)$ identity matrix, i.e. 
\begin{equation}\label{eq:dilation_invariance}
 \big(Y u^{(\lambda)}\big)(z)=\lambda^{\ddY}(Y u)(D_\lambda z), \qquad z=(t,x,v)\in\R\times\R^{2d},\ \lambda>0 ,
\end{equation}
where
\begin{equation}
 u^{(\lambda)} (z)=u\big(D_\lambda(z)\big).
\end{equation}
Note that the exponent $\ddY$ in \eqref{eq:dilation} equals $ps$ in the case of non-local operator $\Lc_{s,p}$ in \eqref{e1_prot-p}, hence we don't impose any restrition to the choice of $\ddY \in ]0, + \infty[$.

In the sequel we will say that $Y$ has formal degree $\ddY$ and that the vector fields $Z_1,\dots,Z_d$ in \eqref{eq:Z_fields} have formal degree $1$, in accordance with the terminology adopted in \cite[Section 2.4]{Bonfiglioli2009}.

Notice that the triplet $\left(\Rdd,\circ,D_{\lambda}\right)$ forms a homogeneous group. Indeed, it is 
well defined the so-called homogeneous norm:
\begin{equation}\label{e7}
 \norm{(t,x,v)}=|t|^{\frac{1}{\ddY}}+|x|^{\frac{1}{\ddY+1}} + |v|,
\end{equation}
and we consider the quasi-distance
\begin{equation}\label{eq:quasi_distance}
 d(z_1, z_2) := \norm{z_2^{-1}\circ z_1}, \qquad z_1,z_2\in\R\times\R^{2d}.
\end{equation}
The following properties directly follow from the definition of the quasi-distance
\begin{equation} \label{eq-dist-invariance}
d(D_\lambda(z_1), D_\lambda(z_2)) =  \lambda \, d(z_1, z_2),  \qquad 
d(z \circ z_1, z \circ  z_2) =  d(z_1, z_2) 
\end{equation}
for every $z, z_1, z_2 \in \Rdd$ and for every $\lambda >0$. Note that $d$ is said \emph{quasi-distance} as the following weaker form of triangular inequality holds for it: there exists a constant $\kappa \ge 1$ such that
\begin{equation}\label{eq:quasi_triangle}
 d(z_1, z_3) \le \kappa \left( d(z_1, z_2) + d(z_2, z_3) \right), \qquad z_1, z_2, z_3 \in \R\times\R^{2d}.
\end{equation}
We recall that in \cite{imbert2018schauder} it is considered the \emph{equivalent} distance 
\begin{equation}\label{imbert2018distance}
\min_{w\in\Rd} \big\{   \max\big(   |t_1-t_2|^{\frac{1}{\ddY}} ,  |x_1-x_2- w (t_1-t_2)|^{\frac{1}{1+\ddY}}  ,   |v_1-w|  ,   | v_2 - w|       \big)    \big\}
, \qquad z_1,z_2\in\R\times\R^{2d}.
\end{equation}

\medskip

We next introduce the notions of intrinsic regularity and intrinsic H\"older space. Let  $\XX$ be a
Lipschitz vector field on $\Rdd$. For any $z\in\Rdd$, we denote by $\t\mapsto e^{\t \XX }(z)$ the
integral curve of $\XX$ defined as the unique solution to
\begin{equation}
\begin{cases}
 \frac{d}{d\t}e^{\t \XX }(z)= \XX \left(e^{\t \XX }(z)\right),\qquad  &\t\in\R, \\
 e^{\t \XX }(z)\vert_{\t=0}= z.
\end{cases}
\end{equation}
In particular, for a vector $h \in \R^d$ we set $Z_h := h_1 Z_1 + \dots + h_d Z_d$ and we find 
we have
\begin{equation}\label{eq:def_curva_integrale_campo}
 e^{\t 
 \ZZ_h }(t,x,v)=(t,x,v+\t h),\qquad
 e^{\t Y }(t,x,v)=(t+\t, x + \t v , v ), \qquad \tau\in\R,
\end{equation}
for any $(t,x,v)\in\Rdd$. 

Next we recall the general notion of Lie differentiability and H\"older regularity.
\begin{definition}\label{def:intrinsic_alpha_Holder3}
Let $\XX$ be a Lipschitz vector field and $u$ be a real-valued function defined in a neighborhood of
$z\in \Rdd$. We say that $u$ is \emph{$\XX$-differentiable} in $z$ if the function $\t\mapsto
u\left(e^{\t \XX }(z)\right)$ is differentiable in $0$. {We will refer to the function $z \mapsto
\XX u (z) :=\frac{d}{d \t} u\left(e^{\t \XX }(z)\right)\big|_{\t=0}$ as \emph{$\XX$-Lie derivative of $u$}, or
simply \emph{Lie derivative of $u$} when the dependence on the field $\XX$ is clear from the
context.}
\end{definition}
\begin{definition}\label{def:intrinsic_alpha_Holder}
Let $u:{\Rdd}\to \R$. Then, for $a\in]0,1]$, we say that $u\in C^{\a}_{\ZZ_i}$, $i=1,\dots,d$, if 
\begin{equation}
\norm{u}_{C^{\a}_{\ZZ_i}}:=
 \sup_{z\in\Rdd\atop \t\in\R\setminus\{0\}} \frac{
 \left|u\left(e^{\t \ZZ_i }(z)\right)-
 u(z)\right|}{|\t|^{\a}} <\infty.
\end{equation}
Moreover, for $\alpha\in]0,\ddY]$, we say that $u\in C^{\a}_{Y}$ if 
\begin{equation}
\norm{u}_{C^{\a}_{Y}}:=
 \sup_{z\in\Rdd\atop \t\in\R\setminus\{0\}} \frac{
 \left|u\left(e^{\t Y }(z)\right)-
 u(z)\right|}{|\t|^{\frac{\a}{\ddY}}} <\infty.
\end{equation}
\end{definition}

Let us introduce the set
\begin{equation} \label{index}
\I = \{ \alpha \in \R \mid \alpha = k + j \ddY\, \text{ with } k,j \in \mathbb{N}_0  \}.
\end{equation}
As we shall see in the sequel, $\alpha \in \I$ represents the regularity indices for which there is a \emph{jump} in the regularity of a function $u$, meaning that new derivatives along $\ZZ_i$ or/and $Y$ appear. 
In particular, the fllowing statements are true. 
\begin{itemize}
\item[-]  If $\alpha\in ]0,1 \wedge \ddY]$ there are no derivatives with respect to any of the vector fields $Z_1,\dots, Z_d$ or $Y$. 
\item[-]  If $\alpha\in ]1 \wedge \ddY, 1 \vee \ddY]$, there are two cases: if $\ddY<1$, then only derivatives along $Y$ appear, up to order $j$ with $\ddY j <\alpha$; if $\ddY>1$, then only derivatives along the vector fields $Z_1,\dots, Z_d$ appear, up to order $k$ with $k <\alpha$. Of course, if $\ddY=1$, the interval is empty and all the fields have the same formal degree (the gradation is the same as in the Heisenberg group). 
\item[-]  If $\a> 1 \vee \ddY$, then there exist derivatives along $Z_1,\dots, Z_d$ and $Y$.
\item[-]  In Theorem \ref{th:main} we show that, if $\a> 1 + \ddY$, then the derivatives $\partial_{x_j}$ also appear for $i=1,\dots,d$. 
\end{itemize}
Now we define the intrinsic H\"older spaces on the homogeneous group $\left(\Rdd,\circ,D_{\lambda}\right)$, by extending the definitions of H\"older spaces given in \cite{PPP16}. 
Namely, our procedure is recursive with respect to the (ordered) indices in $\I$ . In particular, in the second step, which is for $\alpha\in]1 \wedge \ddY, 1 \vee \ddY]$, the definition splits in two different cases depending on whether $\ddY<1$ or $\ddY>1$. If $\ddY=1$, $\alpha\in]1 \wedge \ddY, 1 \vee \ddY]$ is empty and one moves on to the third step.

\begin{definition}\label{def:C_alpha_spaces}
Let $u:{\R\times}\R^{2d}\to \R$ and $\a>0$. Then:
\begin{itemize}
  \item [i)] if $\alpha\in]0,1 \wedge \ddY]$, 
   $u\in C^{\a}$ if 
  the semi-norm
\begin{equation}\label{e9}
  \norm{u}_{C^{\a}}:=\norm{u}_{C^{\a}_{Y}}+\sum_{i=1}^{d} \norm{u}_{C^{\a}_{\ZZ_i}}
\end{equation}
is finite;
\item [ii)]  if $\alpha\in]1 \wedge \ddY, 1 \vee \ddY]$, 
$u\in C^{\a}$ if 
 the semi-norm
\begin{equation}\label{e10}
  \norm{u}_{C^{\a}}:=
 \begin{cases}
\norm{Yu}_{C^{\a-\ddY}}+\sum_{i=1}^{d} \norm{u}_{C^{\a}_{\ZZ_i}}, \qquad \text {if }\ddY<1 \\
  \norm{u}_{C^{\a}_{Y}}+\sum_{i=1}^{d} \norm{\ZZ_i u}_{C^{\a-1}}, \qquad \text {if }\ddY >1
\end{cases}
\end{equation}
is finite. If $\ddY=1$, $\alpha\in]1 \wedge \ddY, 1 \vee \ddY]$ is empty and this case can be skipped;
  \item [iii)] if 
  $\alpha>1 \vee \ddY$, 
  $u\in C^{\a}$ if 
  the semi-norm
\begin{equation}\label{e11}
  \norm{u}_{C^{\a}}:=\norm{Y u}_{C^{\a-\ddY}}+\sum_{i=1}^{d} \norm{\ZZ_i u}_{C^{\a-1}}
\end{equation}
is finite.
\end{itemize}

\end{definition}

\begin{example} Figure 1 below describes the pairs $(k, j \ddY)$ with $k,j \in \N_0$ and $\ddY=1/3$. The filled dots form the set $\I$, the points $\ddY$ and $1 + \ddY$ are highlighted.

\medskip

\begin{center}

\begin{tikzpicture} 
\draw[->](-.5,0)--(5,0);%
\draw[->](0,-.5)--(0,2);%
\foreach \y in {0,1,...,3}
\foreach \x in {0,1,...,3}
\draw (1.5*\x,.5*\y) circle (1.2pt);
\node at (-.15,-.3) {$0$};
\foreach \x in {1,...,3}
\node at (1.5*\x,-.3) {\x};
\node at (5,-.3) {$\I$};
\node at (-.2,.55) {$\ddY$};
\clip (-.05,-.5) rectangle (4.8,2);
\foreach \z in {0,1,...,12}
\draw[dotted] (-1.5 + .5*\z,1.5)--(.5*\z,0);
\foreach \z in {0,1,...,9}
\fill (.5*\z,0) circle (1.2pt);
\node at (.5,.25) {$\ddY$};
\fill[red] (.5,0) circle (1pt);
\node at (2.02,.25) {$1 + \ddY$};
\fill[red] (2,0) circle (1pt);
\end{tikzpicture}

\medskip

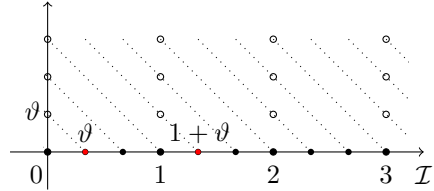
\sc{Fig. 1. The set $\I$ for $\ddY=1/3$.}

\end{center}

\medskip

Let us consider $ u\in C^{\alpha}$. Then we have:
\begin{itemize}
\item[-] If $\alpha \in ]0,1/3]$, there are no derivatives with respect to either $\ZZ_i$ or $Y$. 
\item[-] If $\alpha\in ]1/3, 2/3]$, the Lie derivative $Yu$ exists and belongs to $C^{\alpha-1/3}$. Furthermore, $u$ belongs to $C^{\alpha}_{Z_i}$. Note that the definition is well-posed. Indeed, the index $\alpha-1/3 \in]0,1/3]$ and thus the space $C^{\alpha-1/3}$ has already been defined; also $\alpha \le 1 $ and thus $C^{\alpha}_{Z_i}$ is defined too.
\item[-] If $\alpha\in ]2/3, 1]$, $Yu$ belongs to $C^{\alpha-1/3}$ where $\alpha - 1/3 \in ]1/3 , 2/3]$. In particular there exists $Y^2 u \in C^{\alpha-2/3}$.
\item[-] If $\alpha\in ]1 , 4/3]$, there exist the derivatives along the fields $\ZZ_i$, which are the Euclidean derivatives $\partial_{v_1}u, \cdots, \partial_{v_d}u$. Furthermore, such derivatives belong to $C^{\alpha-1}$ where $\alpha -1 \in ]0, 1/3]$. Also, the Lie derivatives $Yu, Y^2u$ and $Y^3u$ exist and belong to $C^{\alpha-1/3}, C^{\alpha-2/3}$ and $C^{\alpha-1}$, respectively.
\item[-] If $\alpha\in ]4/3, 5/3]$, then $Y^4 u \in C^{\alpha-4/3}$, and $\partial_{v_i}u \in C^{\alpha-1}$. Moreover, 
the mixed derivatives $\partial_{v_i} Y u$ and $Y \partial_{v_i}u$ exist and belong to $C^{\alpha-4/3}$, and so do the commutators $[\partial_{v_i}, Y] u$. Indeed we have $Y u \in C^{\alpha-1/3}$ where $\alpha-1/3 \in ]1, 4/3]$, and $\partial_{v_i}u \in C^{\alpha-1}$ where $\alpha-1 \in ]1/3, 2/3]$. 
\end{itemize}

Once more, the last step is crucial as it can be  proved that $[\partial_{v_i}, Y] u = \partial_{x_i} u$ (see Theorem \ref{th:main} below). By induction, there exist the $n$-th order Euclidean derivatives with respect to $x$ so long as $\alpha > n (\ddY +1)$.  
\end{example}

\begin{example}\label{ex:kappa_3_2} Figure 2 below describes the pairs $(k, j \ddY)$ with $k,j \in \N_0$ and $\ddY=4/3$. The filled dots form the set $\I$.

\medskip

\begin{center}

\begin{tikzpicture} 
\draw[->](-.5,0)--(5.5,0);%
\draw[->](0,-.5)--(0,3.2);%
\foreach \y in {0,1,2}
\foreach \x in {0,1,...,5}
\draw (\x,4/3*\y) circle (1.2pt);
\node at (-.15,-.3) {$0$};
\node at (5.5,-.3) {$\I$};
\foreach \x in {1,...,5}
\node at (\x,-.3) {\x};
\node at (-.2,1.35) {$\ddY$};
\clip (-.05,-.5) rectangle (5.2,2.9);
\foreach \x in {0,...,4}
\foreach \y in {1,...,5}
\draw[dotted] (\x,4/3*\y)--(\x + 4/3*\y,0);
\foreach \x in {0,1,...,5}
\fill (\x,0) circle (1.2pt);
\node at (4/3,.25) {$\ddY$};
\draw (4/3,0) circle (1.2pt);
\fill[red] (4/3,0) circle (1pt);
\node at (7/3,.25) {$1 + \ddY$};
\draw (7/3,0) circle (1.2pt);
\fill[red] (7/3,0) circle (1pt);
\foreach \z in {0,...,7}
\fill (8/3 + \z/3,0) circle (1pt);
\end{tikzpicture}


\medskip

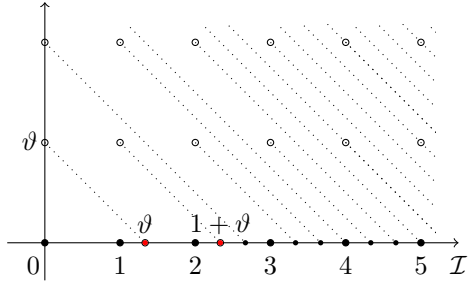
\sc{Fig. 2. The set $\I$ for $\ddY=4/3$.}

\end{center}

\medskip

Let us consider $ u\in C^{\alpha}$. Then we have:
\begin{itemize}
\item[-] If $\alpha \in ]0,1]$, there are no derivatives with respect to either $\ZZ_i$ or $Y$. 
\item[-] If $\alpha\in ]1, 4/3]$, the $1$st order derivatives $\ZZ_i u = \partial_{v_i} u$ exist and belong to $C^{\alpha-1}$. Plus, $u\in C^{\alpha}_Y$. 
\item[-] If $\alpha\in ]4/3 , 2 ]$, the Lie derivative $Yu$ exists and belongs to $C^{\alpha-4/3}$. Moreover, 
$\partial_{v_1}u, \cdots, \partial_{v_d}u\in C^{\alpha-1}$. 
\item[-] If $\alpha\in ]2, 7/3]$, the $2$nd order derivatives $\partial_{v_i v_j}u$, for $i,j=1,\cdots,d$, exist and belong to $C^{\alpha-2}$. Furthermore, $\partial_{v_i}u \in C^{\alpha-1}_Y$ and $Yu   \in C^{\alpha-4/3}$.
\item[-] If $\alpha\in ]7/3, 8/3]$, then the mixed derivatives $\partial_{v_i} Y u$ and $Y \partial_{v_i}u$ exist and belong to $C^{\alpha-7/3}$, and so do the commutators $[\partial_{v_i}, Y] u$. 
Also, $Y u \in C^{\alpha-4/3}$ and $\partial_{v_i v_j}u \in C^{\alpha-2}$.
\end{itemize} 
\end{example}



\medskip In the sequel, $\beta=(\beta_1,\cdots, \beta_d)\in \Ndzero$ will denote a multi-index. As
usual
 $$|\beta|:=\sum_{j=1}^d \beta_j\quad \text{ and }\quad \beta!:=\prod_{j=1}^d \left(\beta_j !\right)$$
are called the length and the factorial of $\b$ respectively. Moreover, for any $x\in\Rd$, we set
$$
x^{\beta}=x_1^{\beta_1}\cdots x_d^{\beta_d}\quad \text{ and }\quad
 \partial^{\beta}=\partial^{\beta}_x=\partial_{x_1}^{\beta_1}\cdots
 \partial_{x_d}^{\beta_d}.
$$
Finally, for $u\in C^{\a}$ we let $T_{\alpha} u(z_0,\cdot)$ its intrinsic Taylor polynomial  around $z_0=(t_0,x_0,v_0)$ defined as
\begin{equation}\label{eq:TAy_pol}
 T_{\alpha} u(z_0;z):=\!\!  \sum_{{0\leq \ddY k + (1+\ddY)|\gamma| + |\beta| < \alpha}} \!\! \frac{   Y^k {\partial_{v}^{\beta} \partial_{x}^{\gamma} }u(z_0)  }{k!\, \gamma! \,\beta! } (t-t_0)^k \big(x-x_0-(t-t_0)v_0 \big)^{\gamma} (v-v_0)^{\beta}.
\end{equation}

\begin{example}
For $d=1$ and $\ddY=4/3$ as in Example \ref{ex:kappa_3_2}, we have
\begin{itemize}
\item[-] If $\alpha \in ]0,1]$, then we set $\alpha_1 :=1$ and we have $T_{\alpha} u(z_0;z) = T_{\alpha_1} u(z_0;z) = u(z_0)$;
\item[-] If $\alpha\in ]1, 4/3]$  ($\alpha_2 :=4/3$), then $T_{\alpha} u(z_0;z) = T_{\alpha_2} u(z_0;z)  = T_{\alpha_1} u(z_0;z) + \big(\partial_v u(z_0)\big) (v-v_0)$;
\item[-] If $\alpha\in ]4/3 , 2 ]$ ($\alpha_3 :=2$), then $T_{\alpha} u(z_0;z) = T_{\alpha_3} u(z_0;z) = T_{\alpha_2} u(z_0;z) + \big(Y u(z_0)\big) (t-t_0)$;
\item[-] If $\alpha\in ]2, 7/3 ]$ ($\alpha_4 :=7/3$), then $T_{\alpha} u(z_0;z) = T_{\alpha_4} u(z_0;z) = T_{\alpha_3} u(z_0;z) + \frac{1}{2} \big(\partial^2_v u(z_0)\big) (v-v_0)^2$;
\item[-] If $\alpha\in ]7/3, 8/3 ]$, then $T_{\alpha} u(z_0;z) 
= T_{\alpha_4} u(z_0;z) + \big(\partial_x u(z_0)\big) (x-x_0 - (t-t_0) v_0) + \big(Y \partial_v u(z_0)\big) (t-t_0)(v-v_0)$. 
\end{itemize}

\end{example}

{ 
\begin{theorem}\label{th:main}
For any  $\alpha >0$ and for any $u\in C^{\a}$ the following statements are true:
\begin{description}
 \item [{\sc H\"older spaces characterization.}] For any $k\in\N_0$ and $\gamma,\beta\in\N_0^d$ with ${0\leq \ddY k + (1+\ddY)|\gamma| + |\beta| <\alpha}$, the derivatives 
$ \ Y^k \partial_v^{\beta} \partial_x^{\gamma} u$ exist and 
\begin{equation}\label{eq:maintheorem_part1_global}
 \| Y^k  \partial_v^{\beta} \partial_x^{\gamma}u \|_{C^{\alpha-\ddY k - (1+\ddY)|\gamma| - |\beta|}}  \leq \| u \|_{C^{\alpha}} .
\end{equation}
 \item [{\sc Taylor formula.}] There exists a positive constant $c>0$, only dependent on $\alpha$, 
 such that
\begin{equation}\label{eq:estim_tay_n}
 \left|u(z)-T_{\alpha} u(z_0;z)\right|\le  c \|u\|_{C^{\a}} \|z_0^{-1}\circ z\|^{\a},\qquad z,z_0\in \Rdd.
\end{equation}
\end{description}
\end{theorem}
}

\begin{remark} \label{rcaa}
If $\alpha\in]0, 1 \wedge \ddY ]$, estimate \eqref{eq:estim_tay_n} restores the definition of H\"older continuous function given in Definition 1.2 in \cite{Polidoro94}
\begin{equation*}
 \left|u(z)-u(z_0)\right|\le c \|u\|_{C^{\a}} \norm{z_0^{-1}\circ z}^{\a},\qquad z,z_0\in\Rdd.
\end{equation*}
For a comparison between intrinsic and Euclidean H\"older continuity we refer to Proposition 2.1 in \cite{Polidoro94}. 

\end{remark}

\begin{remark} \label{rem-dx-dv}
By \eqref{eq:maintheorem_part1_global}-\eqref{eq:estim_tay_n}, it is straightforward to see that $\partial_v^{\beta} \partial_x^{\gamma}u$ and $ \partial_x^{\gamma} \partial_v^{\beta}u$ are both continuous on $\R\times\R^{2d}$ for any $\gamma,\beta\in\N_0^d$ with ${0\leq  (1+\ddY)|\gamma| + |\beta| <\alpha}$, 
hence they agree. By contrast, $Y^k$ does not commute with $\partial_v^{\beta}u$.
\end{remark}


\section{Preliminaries}
The method of the proof relies on the construction of a finite sequence of integral curves of the vector fields $Y, Z_1, \dots, Z_d$ which steer a point $z_0 = (t_0, x_0, v_0)$ to any other point $z = (t, x, v)$. We then rely on the usual Taylor expansion of the functions $\tau \mapsto u\left( e^{\tau Y}(z_0) \right)$ and $\tau \mapsto u\left( e^{\tau Z_h}(z_0) \right)$ introduced in \eqref{eq:def_curva_integrale_campo}, to obtain a good approximation of $u$ near $z_0$. The approximation of $u$ along the integral curves of the commutators $[Y, Z_i]$ is obtained by using a rather classical technique from control theory. 
We first consider the approximation along the integral curve $e^{\tau Y}(z)$.


\begin{remark}\label{rem:eq:mean_value_Yn}
Let $n\in\N_0$, $\gamma \in ] 0, \ddY ]$ and $u\in C^{\ddY n+\gamma}$. 
Then, by Definition \ref{def:C_alpha_spaces}, we have $Y^{j} u\in C^{\ddY (n-j)+\gamma}_{Y}$ with $j=1,\dots,n$. Therefore, by mean-value theorem along the vector field $Y$, for any $z=(t,x,v)\in\Rdd$ and $\tau \in\R$, there exists $\delta \in ]0,1[$ such that
\begin{equation}\label{eq:mean_value_Yn_loc}
 u\big(e^{\tau Y}(z)\big)-u(z)-\sum_{j=1}^{n} \frac{\tau^j}{j!}Y^j u(z)= \frac{\tau^{n}}{n!}\left( Y^{n} u\big(e^{{\delta \tau}Y}(z)\big)- Y^{n}u(z)\right), 
\end{equation}
and thus, Definition \ref{def:intrinsic_alpha_Holder} yields
\begin{align}\label{eq:mean_value_Yn_B}
 \big| u \big( e^{\tau Y}(z) \big) - T_{\ddY n+\gamma} u\big(z;e^{\tau Y}(z)\big)\big| & =  \Big| u\big(e^{\tau Y}(z)\big)-u(z)-\sum_{j=1}^{n} \frac{\tau^j}{j!}Y^iu(z) \Big| \\
&  \leq \frac{1}{n!} 
 \|u\|_{C^{\ddY  n+\gamma}}|\tau|^{n+\frac{\gamma}{\ddY }},\qquad \tau\in\R,\quad z\in \R\times\R^{2d}.
\end{align}
\end{remark}

Since the vector fields $Z_1, \dots, Z_d$ have constant coefficients, the usual Euclidean Taylor theorem with Lagrange reminder plainly gives the following result.

\begin{remark}\label{rem:eq:mean_value_Zn}
Let $n\in\N_0$, $\gamma \in ]0,1]$ and $u\in C^{n+\gamma}$. Then, by Definition \ref{def:C_alpha_spaces}, we have $\partial_{v}^{\beta} u\in C^{\gamma + n - |\beta|}_{Z_i}$ for any $\beta\in\N_0^{d}$ with $|\beta|\leq n$ and $i=1,\cdots,d$. Therefore, recalling Definition \ref{def:intrinsic_alpha_Holder}, mean-value theorem yields
\begin{equation}\label{eq:mean_value_Zn_B}
 \big| u ( t, x, v + h  ) - T_{n+\gamma} u\big(z; t, x, v + h \big)\big|   \leq \frac{1}{n!} \|u\|_{C^{n+\gamma}}|h|^{n+ \gamma},\qquad z = (t,x,v)\in \R\times\R^{2d}, \quad h\in\R^d.
\end{equation}
\end{remark}

In view of Theorem \ref{th:main}, we consider two points $z_0, z \in \R^{1+2d}$ and we note that Remarks \ref{rem:eq:mean_value_Yn} and \ref{rem:eq:mean_value_Zn} provide us with a bound of $|u(z_0) - u(z_1)|$ in terms of $\| z_0^{-1} \circ z \|$, where $z_1$ is a specific point of $\R^{1+2d}$ whose components $t$ and $v$ agree with the components of $z$. The following picture illustrates the situation in the case $d=1$ and $z = (0,0,0)$. 

\medskip

\begin{center}
 
\begin{tikzpicture} 
\draw[->](-.5,1)--(3.3,1);%
\node at (3.3,1.2) {$v$};
\draw[->](0,.4)--(0,2);%
\node at (.2,2) {$t$};
\draw[->](.5,1.25)--(-1.5,.25);%
\node at (-1.4,.5) {$x$};
\node at (1.2,1.6) {$z_0$};
\draw [brickred,line width=.8pt] (1,1.5)--(2,.5);
\node at (2.9,.6) {$e^{- t_0 Y}(z_0)$};
\draw [forestgreen,line width=.8pt] (2,.5)--(-1,.5);
\node at (.6,.1) {$z_1 = e^{- v_0 Z_1}e^{- t_0 Y}(z_0)$};
\fill (1,1.5) circle (1pt);
\fill (2,.5) circle (1pt);
\fill (-1,.5) circle (1pt);
\end{tikzpicture}

\medskip

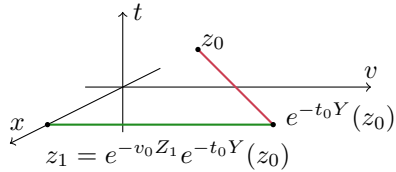
\sc{Fig. 3}

\end{center}

\medskip

We finally recover the regularity of the function $u$ with respect to the variable $x$ by the usual trajectory defined as concatenations of integral curves of $Z_1, \dots, Z_m$ and $Y$, as Figure 4 below shows.

\medskip

\begin{center}
 
\begin{tikzpicture} 
\draw[->](-4,1)--(2.4,1);%
\node at (2.3,.8) {$v$};
\draw[->](0,.5)--(0,2.5);%
\node at (.2,2.5) {$t$};
\draw[->](.25,1.125)--(-1.5,.25);%
\node at (-1.4,0) {$x$};
\node at (.1,.3) {$z_1 = (0,x_1,0)$};
\draw [forestgreen,line width=.8pt] (-1,.5)--(-4,.5);
\node at (-3.4,.2) {$z_2 = e^{\tau Z}(z_1)$};
\draw [brickred,line width=.8pt] (-4,.5)--(-3,1.75);
\draw[dotted](-4,.5)--(-3,1)--(-3,1.75);%
\node at (-2.3,2) {$z_3 = e^{\tau^{\ddY }Y}(z_2)$};
\draw [forestgreen,line width=.8pt] (-3,1.75)--(0,1.75);
\node at (1.2,1.9) {$z_4 = e^{-\tau Z}(z_3)$};
\draw [brickred,line width=.8pt] (0,1.75)--(0,1);
\node at (1.4,1.3) {$z = e^{-\tau^{\ddY }Y}(z_4)$};
\fill (-1,.5) circle (1pt);
\fill (-4,.5) circle (1pt);
\fill (-3,1.75) circle (1pt);
\fill (0,1.75) circle (1pt);
\fill (0,1) circle (1pt);
\end{tikzpicture}

\medskip

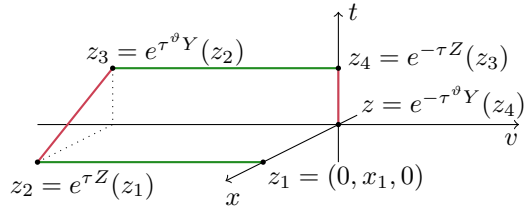
\sc{Fig. 4}

\end{center}

The following Proposition provides us with a quantitative estimate of the increment of a function $u\in C^{\alpha}$ with respect to the variable $x$.

\begin{proposition}\label{prop:holder_x}
Let $u\in C^{\alpha}$ with $\alpha\in ]0, 1 + \ddY]$, then 
\begin{equation}\label{eq:claim}
|u(t,x,v) - u(t,x+h,v)| \leq c \, \|u\|_{C^{\a}} | h |^{\frac{\a}{\ddY+1}},\qquad z=(t,x,v)\in \Rdd, \quad h\in\Rd.
\end{equation}
\end{proposition}
\begin{proof}
Assume $h\neq 0$ and set $w:=h/|h|$, $\tau:= |h|^{\frac{1}{\ddY+1}}$.
Recall the notation in \eqref{eq:def_curva_integrale_campo}, put 
\begin{equation} \label{eq:gamma_k1_bis}
 z_2 = e^{\tau Z_w}(t,x,v), \qquad z_3 = e^{\tau^{\ddY }Y}(z_2) \qquad z_4 = e^{- \tau Z_w}(z_3), 
\end{equation}
and note that
\begin{equation} \label{eq:gamma_k1_ter}
e^{-\tau^{\ddY }Y}(z_4) = \big(t, x + \tau^{\ddY  + 1} w , v \big) = \big(t, x + h , v \big).
\end{equation}
With this notation, we express the left hand side of \eqref{eq:claim} as follows
\begin{equation}\label{eq-boxes}
\begin{split}
u(t,x+ h ,v) - u (t,x,v) 
&=\boxed{u(t, x+h,v) - u (z_4) - \sum_{\ddY \le \ddY j < \alpha} \frac{(-\tau^\ddY )^{j}}{j!} Y^j u(z_4)}_{(1)}\\
&+\boxed{u(z_4) - u (z_3) - \sum_{1 \le k < \alpha} \frac{(-\tau )^{k} Z_w^k u(z_3)}{k!}}_{(2)} \\
&+\boxed{u(z_3)-u(z_2) + \sum_{\ddY \le \ddY j < \alpha} \frac{(-\tau^\ddY )^{j}}{j!} Y^j u(z_3)}_{(3)} \\
&+\boxed{u (z_2) - u(t, x,v) + \sum_{1 \le k < \alpha} \frac{(-\tau )^{k}}{k!}Z_w^k u(z_2)}_{(4)} \\
&+\boxed{\sum_{\ddY \le \ddY j < \alpha} \frac{(-\tau^\ddY )^{j}}{j!} \left(Y^j u(z_4) - Y^j u(z_3) \right)}_{(5)}\\
&+\boxed{\sum_{1 \le k < \alpha} \frac{(-\tau )^{k}}{k!} \left(Z_w^k u(z_2) - Z_w^k u(z_3) \right)}_{(6)} 
 =: I_1 + I_2 + I_3 + I_4 + I_5 + I_6,
\end{split}
\end{equation}
where the indices $j,k$ in the above summations are non-negative integers. 
Note that, if $\alpha \le \ddY$, then no derivatives $Y^j u$ apperar in $I_1, I_3$ and the term $I_5$ doesn't appear. Thus, in the following we consider separately this case and $\ddY < \alpha \le 1 + \ddY$. Analogously, if $\alpha \le 1$, no derivatives of the form $Z_w^k u$ apperar in $I_2, I_4$. We claim that, in every case, the terms $I_1, \dots, I_6$ are bounded by $\|u\|_{C^{\a}} | h |^{\frac{\a}{\ddY+1}}$, up to multiplying by a positive constant. 

Consider first $I_1, I_3$ and $I_5$ with $\ddY < \alpha \le 1 + \ddY$. We apply Remark \ref{rem:eq:mean_value_Yn} to  $I_1$ and $I_3$, and we find
\begin{equation}\label{eq:estimate_I13}
 |I_1|, |I_3| \le \|u\|_{C^{\a}} | h |^{\frac{\a}{\ddY+1}}.
\end{equation}
Moreover, $\alpha \le 1 + \ddY$ also yields $Y^j u \in C_{Z_i}^{\alpha - j \ddY}$, with $0 < \alpha - j \ddY \le 1$ for every $j$ such that $\ddY \le \ddY j < \alpha$. Then
\begin{equation}\label{eq:mean_value_Yn_prop_2}
 |\tau|^{j \ddY} \, \left|  Y^j u(z_4) - Y^j u(z_3)  \right| \le 
  \|u\|_{C^{\a}} | h |^{\frac{\a}{\ddY+1}}, \qquad  \ddY \le \ddY j < \alpha, 
\end{equation}
because of the very definition of $C^\alpha$ space. As a consequence,
\begin{equation}\label{eq:estimate_I5}
 |I_5| \le e \, \|u\|_{C^{\a}} | h |^{\frac{\a}{\ddY+1}}.
\end{equation}
If $0 < \alpha \le \ddY$ instead, we set $I_5 = 0$, so that \eqref{eq:estimate_I5} trivially holds. Moreover, \eqref{eq:estimate_I13} follows again from the definiton of $C^\alpha$. In both cases \eqref{eq:estimate_I13} and \eqref{eq:estimate_I5} hold.

%

The argument for $I_2, I_4$ and $I_6$ is analogous. If $1 < \alpha \le 1 + \ddY$, then Remark \ref{rem:eq:mean_value_Zn} yields
\begin{equation}\label{eq:estimate_I24}
 |I_2|, |I_4| \le \|u\|_{C^{\a}} | h |^{\frac{\a}{\ddY+1}},
\end{equation}
and $\alpha \le \ddY + 1$ implies $Z_w^k u \in C_{Y}^{\alpha - k}$, with $ 0 < \alpha - k \le \ddY$ for every $k \in \N$ such that $k < \alpha$. Then 
\begin{equation}\label{eq:mean_value_Zn_prop_2}
 |\tau|^{k} \, \left|  Z^k_w u(z_2) - Z^k_w u(z_3)  \right| \le 
  \|u\|_{C^{\a}} | h |^{\frac{\a}{\ddY+1}}, \qquad  1 \le k < \alpha,
\end{equation}
and this inequality yields
\begin{equation}\label{eq:estimate_I6}
 |I_6| \le e \, \|u\|_{C^{\a}} | h |^{\frac{\a}{\ddY+1}}.
\end{equation}
If $0 < \alpha \le 1$ we set $I_6 = 0$ and the bound \eqref{eq:estimate_I24} holds by the definition of $C^\alpha$. 

The conclusion of the proof of \eqref{eq:claim} is then a direct consequence of \eqref{eq:estimate_I13}-\eqref{eq:estimate_I5}  and \eqref{eq:estimate_I24}-\eqref{eq:estimate_I6}.
\end{proof}


\section{Proofs}

The proof of Theorem \ref{th:main} is divided into two steps: 
\begin{itemize}
\item[1.] (preliminary result) assuming true the first statement of Theorem \ref{th:main} for some $\alpha> 0$, we prove that also the second statement is true for the same $\alpha$;
\item[2.] we prove the first statement of Theorem \ref{th:main} for a general $\alpha> 0$, 
by means of a suitable induction procedure.
\end{itemize}


\subsection{The preliminary result}
We assume the first statement of Theorem \ref{th:main} to be true for fixed $\alpha>0$ and prove the second one. 

\vspace{2pt}
\noindent\underline{Step I:} We fix 
\begin{equation}
z_1 = (t_1, x_1, v_1)= e^{( t -t_0)Y}(z_0)= (t, x_0 +(t-t_0) v_0, v_0)
\end{equation}
 and we prove that
\begin{equation}\label{eq:estim_tay_n_space}
\left| u(z)-T_{\alpha}  u(z_1,z)\right|\le  c \|u\|_{C^{\a}} \big(|x-x_1|^{\frac{1}{\ddY +1}} + |v-v_1|\big)^{\a} = c \|u\|_{C^{\a}} \, \|z_1^{-1}\circ z \|^{ \a}.
\end{equation}
Set $z_2 := (t_1, x_1, v) $. We rearrange the first term as 
\begin{align}
u(z)-T_{\alpha}  u(z_1,z) &= \boxed{ u(z) - T_{\alpha}  u(z_2,z)}_{(1)} + \boxed{ T_{\alpha}  u(z_2,z) -T_{\alpha}  u(z_1,z)}_{(2)} =: I_1 + I_2.
\end{align}
We have 
\begin{equation}
I_1 = u(z) - \sum_{(1+ \ddY) |\gamma| < \alpha} \frac{1}{\gamma!} \partial^{\gamma}_x u(z_2) (x - x_1)^{\gamma}.
\end{equation}
By assumption, for any multi-index $\gamma$ in the summation above with highest order, namely $|\gamma | = \lceil \alpha/(1+ \ddY)\rceil - 1$, we have
$\partial_x^{ \gamma} u \in C^{\alpha-(1+\ddY)|\gamma|}$. Therefore, Proposition \ref{prop:holder_x} together with Euclidean Taylor formula yield
\begin{equation}
|I_1| \leq c\, \|u\|_{C^{\a}} \,  |x - x_1|^{\frac{\alpha}{1+ \ddY}}  \leq  c\, \|u\|_{C^{\a}} \, \|z_1^{-1}\circ z \|^{ \a}.
\end{equation}
For the second term, we have
\begin{equation}
I_2 = \sum_{(1+ \ddY) |\gamma| < \alpha}  \frac{1}{ \gamma!}(x - x_1)^{\gamma}\big( \partial_x^{\gamma}  u(z_2) - T_{\alpha-(1+ \ddY) |\gamma|} (\partial_x^{\gamma} u)(z_1,z_2)  \big).
\end{equation}
By assumption $\partial_x^{\gamma} u \in C^{\alpha-(1+ \ddY) |\gamma|}$, and thus, by applying Remark \ref{rem:eq:mean_value_Zn} with $h=v-v_0$ to each term in the summation we obtain 
\begin{equation}
|I_2| \leq c\, \|u\|_{C^{\a}} \!\!\! \sum_{(1+ \ddY) |\gamma| < \alpha} |x - x_1|^{\gamma} |v - v_0|^{\alpha-(1+ \ddY) |\gamma|} \leq c\, \|u\|_{C^{\a}} \, \|z_1^{-1}\circ z \|^{ \a}.
\end{equation}



\vspace{2pt}

\noindent\underline{Step II:} we conclude the proof of \eqref{eq:estim_tay_n}. 
Note that
\begin{equation}
z_1^{-1}\circ z= (0, x-x_0-(t- t_0)v_0,v- v_0), \quad
z_0^{-1}\circ z= (t-t_0,x-x_0-(t- t_0)v_0,v- v_0),
\end{equation}
so that $\|z_1^{-1}\circ z\| \leq \|z_0^{-1}\circ z\|$.
Therefore, by \eqref{eq:estim_tay_n_space}, we have
\begin{equation}
|u(z)-T_\a(z_0, z)| = |u(z)-T_\a(z_1, z)| +  |T_\a(z_1, z)-T_\a(z_0, z)| \le c\, \|u\|_{C^\a} \|z_0^{-1}\circ z\|^\a+
|T_\a(z_1, z)-T_\a(z_0, z)|.
\end{equation}
Rearranging the Taylor polynomials we can write
\begin{align}
T_\a(z_1, z)-T_\a(z_0, z)= 
\sum_{{(1+\ddY )|\gamma| + |\beta| < \alpha}} \!\! 
\frac 1 {\gamma! \,\beta! }& \Big(   \partial_{x}^{\gamma}\partial_{v}^{\beta}u(z_1)  -T_{\a -(1+\ddY )|\gamma| - |\beta| } ( \partial_{x}^{\gamma}\partial_{v}^{\beta} u)(z_0,z_1)\Big)\\
&\times
  \big(x-x_0-(t-t_0)v_0 \big)^{\gamma} (v-v_0)^{\beta}.
\end{align}
By Remarks \ref{rem-dx-dv} and \ref{rem:eq:mean_value_Yn} we obtain
\begin{equation}
|T_\a(z_1, z)-T_\a(z_0, z)|\le c\, \|u\|_{C^\a} \|z_0^{-1}\circ z\|^\a.
\end{equation}
This concludes the proof of the second statement of Theorem \ref{th:main}, assuming that the first one holds true.

%
%
%
%
%


\subsection{Proof of Theorem \ref{th:main}}

If $\alpha\leq 1+ \ddY$, then \eqref{eq:maintheorem_part1_global} does not contain derivatives with respect to $x$, thus it stems from Definition \ref{def:C_alpha_spaces}. 

\vspace{2pt}
We prove  \eqref{eq:maintheorem_part1_global} for $\alpha > 1+ \ddY$. Notice that it is enough to show that, for any $ i=1,\dots, d$, we have
\begin{equation}\label{eq:comm_equal_deriv_x}
 \partial_{x_i} u = [Z_i,Y] u. 
\end{equation}
Indeed, with \eqref{eq:comm_equal_deriv_x} at hand, 
\begin{equation}\label{eq:der_x}
\partial_{x_i} u \in C^{\alpha- (\ddY  + 1)}
\end{equation}
and the whole \eqref{eq:maintheorem_part1_global} follows, once more, by Definition \ref{def:C_alpha_spaces}, {and Remark \ref{rem-dx-dv}, combined with a plain induction argument.} 

To show \eqref{eq:comm_equal_deriv_x}, we consider any $\alpha > 1+ \ddY$, we let $n:= \lceil \tfrac{\alpha}{\ddY + 1} \rceil - 1$, so that $\alpha \in ] n(1+ \ddY), (n+1)(1+ \ddY)]$, and we assume \eqref{eq:maintheorem_part1_global} true for $\bar \alpha := \alpha -(\ddY +1)$. We claim that
\begin{equation}\label{eq:claim_step_2}
\frac{u(t,x+\d \, e_i ,v) - u (t,x,v)}{\d} \to [Z_i,Y] u (t,x,v)
\quad \text{as} \quad \d \to 0, \qquad  (t,x,v) \in\Rdd.
\end{equation}
Here $e_i$ denotes the $i$-th element of the canonical basis of $\Rd$. To prove \eqref{eq:claim_step_2} we rely on an arugment similar to that used in the proof of Proposition \ref{prop:holder_x}. We fix $(t,x,v) \in\Rdd, i \in \{1,\dots, d\}, \d \ne 0$, and we recall the notation in \eqref{eq:def_curva_integrale_campo}. {If $\delta >0$ we set $w := e_i$, otherwise we set $w := - e_i$}. For simplicity we only consider the case $\delta>0$, the case $\delta<0$ being completely analogous. Set $\tau: = \d^\frac{1}{1 + \ddY}$, 
\begin{equation} \label{eq:gamma_k1_step_2}
 z_2 := {e^{\tau Z_w}(t,x,v) = e^{\tau Z_i}(t,x,v)}, 
 \qquad z_3 := e^{\tau^{\ddY }Y}(z_2) 
 \qquad {z_4 := e^{- \tau Z_w}(z_3) = e^{- \tau Z_i}(z_3),} 
\end{equation}
and note that we have
\begin{equation} \label{eq:gamma_k1_step_2+}
e^{-\tau^{\ddY }Y}(z_4) = \big(t, x + \tau^{\ddY  + 1} e_i , v \big)
= \big(t, x + \d \, e_i , v \big).
\end{equation}

We proceed as in the proof of Proposition \ref{prop:holder_x}, {the unique difference being that $\alpha > \ddY + 1$ here}. We have
\begin{equation}\label{eq-boxes-2}
\begin{split}
u(t,x+\d \, e_i ,v) - u (t,x,v) 
&=\boxed{u(t, x+\d \, e_i,v) - u (z_4) - \sum_{\ddY \le j \ddY < \alpha} \frac{(-\tau^\ddY )^{j}}{j!} Y^j u(z_4)}_{(1)}\\
&+\boxed{u(z_4) - u (z_3) - \sum_{1 \le k < \alpha} \frac{(-\tau )^{k} Z_i^k u(z_3)}{k!}}_{(2)} \\
&+\boxed{u(z_3)-u(z_2) + \sum_{\ddY \le j \ddY < \alpha} \frac{(-\tau^\ddY )^{j}}{j!} Y^j u(z_3)}_{(3)} \\
&+\boxed{u (z_2) - u(t, x,v) + \sum_{1 \le k < \alpha} \frac{(-\tau )^{k}}{k!}Z_i^k u(z_2)}_{(4)} \\
&+ \boxed{\sum_{\ddY \le j \ddY < \alpha} \frac{(-\tau^\ddY )^{j}}{j!} \left(Y^j u(z_4) - Y^j u(z_3) \right)}_{(5)}\\
&+\boxed{\sum_{1 \le k < \alpha} \frac{(-\tau )^{k}}{k!} \left(Z_i^k u(z_3) - Z_i^k u(z_2) \right)}_{(6)} =: I_1 + I_2 + I_3 + I_4 + I_5 + I_6.
\end{split}
\end{equation}
{Note that all the sums in the above boxes are not void, since $\alpha > \ddY + 1$.}
By Remark \ref{rem:eq:mean_value_Yn} and \ref{rem:eq:mean_value_Zn} we find that 
\begin{equation}\label{eq-boxes-1-4}
|I_1| + |I_2| + |I_3| + |I_4| \le c_ \alpha \|u\|_{C^{\a}} | \d |^{\frac{\a}{\ddY+1}},
\end{equation}
for some positive constant $c_\alpha$. Concerning $I_5$ and $I_6$, we have
\begin{equation}\label{eq-boxes-5}
\begin{split}
I_5 + I_6
&=\boxed{\sum_{\ddY \le j \ddY < \alpha} \frac{(-\tau^\ddY )^{j}}{j!} \bigg(Y^j u(z_4) - Y^j u(z_3) - 
\sum_{0 <  k < \alpha - j \ddY} \frac{(-\tau )^{k}}{k!}Z_i^k Y^j u(z_3)\bigg)}_{(5.1)}\\
&+\boxed{\sum_{1 \le k < \alpha} \frac{(-\tau )^{k}}{k!} \bigg(Z_i^k u(z_3) - Z_i^k u(z_2) + 
\sum_{0 < j \ddY < \alpha - k} \frac{(-\tau^\ddY )^{j}}{j!} Y^j Z_i^k  u(z_3)\bigg)}_{(6.1)}\\
&+\boxed{\sum_{k >0,j>0 \atop j \ddY + k < \alpha}  \frac{(-\tau^\ddY )^{j}}{j!} \frac{(-\tau )^{k}}{k!}
\big( v_{jk}(z_3) - T_{\alpha_{jk}}v_{jk}(z;z_3) \big)}_{(7)} \\
&+\boxed{\sum_{k >0,j>0 \atop j \ddY + k < \alpha}  \frac{(-\tau^\ddY )^{j}}{j!} \frac{(-\tau )^{k}}{k!}
T_{\alpha_{jk}}v_{jk}(z;z_3)}_{(8)}=: I_{5.1} + I_{6.1} + I_7 + I_8.
\end{split}
\end{equation}
where 
\begin{equation}
 \alpha_{jk} := \alpha -(j \ddY +k), \qquad \text{and} \qquad v_{jk} := [Z_i^k, Y^j]u. 
\end{equation}

Consider $I_{5.1}$ first. 
By Remark \ref{rem:eq:mean_value_Zn}, 
we obtain
\begin{equation} 
 \left| Y^j u(z_4) - Y^j u(z_3) -
 \sum_{0 <  k < \alpha - j \ddY} \frac{(-\tau )^{k}}{k!}Z_i^k Y^j u(z_3) \right| \le 
  \|u\|_{C^{\a}} |\tau|^{\alpha - j \ddY}.
\end{equation}
The same argument applies to $I_{6.1}$, in this case we use Remark \ref{rem:eq:mean_value_Yn}. By collecting the above inequalities we find that there exists a positive constant $c'_\alpha$ such that 
\begin{equation}\label{eq-boxes-5-6.1}
|I_{5.1}| + |I_{6.1}| \le c'_\alpha \|u\|_{C^{\a}} | \d |^{\frac{\a}{\ddY+1}}.
\end{equation}

Concerning $I_{7}$, we first note that, {by Definition \ref{def:C_alpha_spaces},} $v_{jk} \in C^{\alpha_{jk}}$. Moreover, only terms with $0 < \alpha_{jk} \le \bar \alpha = \alpha - (\ddY + 1)$ appear there. Then, because of the induction hypothesis{, in particular by \eqref{eq:estim_tay_n}}, we have
\begin{equation}\label{eq:estim_vjk}
 \left|v_{jk}(z_3) - T_{\alpha_{jk}}v_{jk}(z;z_3) \right|\le c \|v_{jk}\|_{C^{\a_{jk}}} \|z^{-1}\circ z_3\|^{\a_{jk}}.
\end{equation}
We then conclude that there exists a positive constant $c''_\alpha$ such that 
\begin{equation}\label{eq-boxes-567}
 |I_7| \le c''_\alpha \|u\|_{C^\alpha}  | \d |^{\frac{\a}{\ddY+1}}. 
\end{equation}
We are left with the term $I_8$. We note that
\begin{equation}
I_8 = \d\, [Z_i, Y]u(z) + \!\!\!\!\!\!\!\!\!\!\!\!\!\!\!\!\!\!\!\!\!\!\!\!
\sum_{k >0,j>0 \atop \ddY + 1 < m' (\ddY +1) + (j'+j) \ddY + k' +k < \alpha} 
\!\!\!\!\!\!\!\!\!\!\!\!\!\!\!\!\!\!\!\!\!\!\!\!
\d^{\frac{m' (\ddY +1) + (j'+j) \ddY + k' +k}{\ddY +1}} 
\big( \bar c_{m'j'k'jk} \partial_{x_i}^{m'}Y^{j'} Z_i^{k'+k} Y^j u + 
\tilde c_{m'j'k'jk} \partial_{x_i}^{m'}Y^{j'} Z_i^{k'} Y^j Z_i^k u \big)(z), 
\end{equation}
where every constant $\bar c_{m'j'k'jk}$ and $\tilde c_{m'j'k'jk}$ is obtained by collecting the coefficients of the Taylor polynomials of the $v_{jk}$ functions. As a consequence of the fact that $\frac{m' (\ddY +1) + (j'+j) \ddY + k' +k}{\ddY +1} > 1$ for every term of the above sum, we finally obtain
\begin{equation}\label{eq-boxes-8}
\frac{I_8}{\d} \to [Z_i, Y]u(z) \qquad \text{as} \quad \d \to 0.
\end{equation}
Moreover, since $1 + \ddY < \alpha \le 2(1 + \ddY)$, \eqref{eq-boxes-1-4}, \eqref{eq-boxes-5} and \eqref{eq-boxes-567} yield
\begin{equation} \label{eq-boxes-9}
 \frac{I_1 + I_2 + I_3 + I_4}{\d} \to 0, \qquad  \frac{I_{5.1} + I_{6.1}}{\d} \to 0, 
 \qquad \frac{I_7}{\d} \to 0 \qquad \text{as} \quad \d \to 0. 
\end{equation}
The proof of \eqref{eq:claim_step_2} follows from \eqref{eq-boxes-8} and \eqref{eq-boxes-9}.



%
%
%
%


\section{Extensions: non-homogeneous and local case}
In some applications, the operator $Y$ in \eqref{e1} appears in a more general form than \eqref{eq:Y}, namely
\begin{equation}
Y =\big\langle   B (x,v)^\top , \nabla_{(x,v)}   \big\rangle  + \partial_t , \qquad (x,v)\in\R^{2d},
\end{equation}
with $B$ being a $(2 d\times 2 d)$-matrix with real entries that admits the following block decomposition:
\begin{equation}
B =\left( \begin{matrix}
B_{11} & B_{12} \\
B_{21} & B_{22}
\end{matrix}\right), 
\end{equation}
where each block is a $(d\times d)$-matrix and $B_{12}$ has rank $d$. For instance, in the kinetic model originally introduced in (\cite{langevin1908theorie}), the term $B_{12}$ is the $d \times d$ identity matrix, while $B_{22}$ is non-null and depends on the viscosity of the liquid. In mathematical finance, $B_{22}$ may represent an interest rate in the pricing of path-dependent derivatives.

Note that the H\"ormander condition \eqref{eq:hormander} is still satisfied: in particular, \eqref{eq:commutator} becomes
\begin{equation}\label{eq:commutator2}
    B_{12} \nabla_{x}  = [\nabla_{v} , Y] - B_{22} \nabla_{v}.
\end{equation}
The integral curve of the vector field $Y$ now reads as
\begin{equation}\label{eq:def_curva_integrale_campo_non}
 e^{\t Y }(t,x,v)=\big(t+\t, e^{\t B} (x,v)^\top \big), \qquad (t,x,v)\in\Rdd, \quad \t\in\R.
\end{equation}
Accordingly, the relevant non-commutative group law ``$\circ$'' in \eqref{eq:translation_non} is replaced by 
\begin{equation}\label{eq:translation_non}
 z_1\circ z_2 = \left(t_1+t_2,  (x_2,v_2)^\top + e^{t_2 B} (x_1,v_1)^\top  \right),\qquad z_1=(t_1,x_1,v_1),z_2=(t_2,x_2,v_2)\in \Rdd.
\end{equation}
Notice that $(\R\times\R^{2d},\circ)$ remains a group, with the identity and the inverse elements that now read as
\begin{equation}
{\text{Id}=(0,0,0)}, \qquad (t,x,v)^{-1}=\big(-t, - e^{-t  B} (x,v)^\top \big).
\end{equation}
In particular, we have 
\begin{equation}\label{eq:translat_non}
(t_0,x_0,v_0)^{-1} \circ (t,x,v) = \big(t- t_0, (x,v)^\top - e^{(t-t_0) B} (x_0,v_0)^\top \big) .
\end{equation}
Despite the fact that $\Lc_s$ is no longer homogeneous with respect to the dilations $\left(D_\lambda\right)_{\l>0}$ in \eqref{eq:dilation}, the homogeneous norm \eqref{e7} remains well behaved with respect to the stratification induced by the H\"ormander condition, in particular by \eqref{eq:commutator2}. 

A local version of Theorem \ref{th:main} can be proved in this more general setting, with the H\"older spaces $C^{\a}$ begin exactly as in Definition \ref{def:C_alpha_spaces}. However, it is more natural to localize the definition of H\"older spaces as well. We follow the approach in \cite{PPP16}, \cite{pagliarani2017intrinsic} for the case $\ddY=2$. Let $\O$ be a domain in $\Rdd$. For any $z\in\O$ we set
  $$\d_{z}:=\sup\left\{\bar{\d}\in\,]0,1]\mid e^{\d Z_1}(z),\cdots,e^{\d Z_{d}}(z),e^{\d Y}(z)\in\O\text{ for any }\d\in [-\bar{\d},\bar{\d}]\right\}.$$
If $\O_{0}$ is a bounded domain with $\overline{\O}_{0}\subseteq\O$, we set
  $\d_{\O_{0}}:=\min_{z\in \overline{\O}_{0}}\d_{z}.$
Note that $\d_{\O_{0}}\in\,]0,1]$. 
Now we first localize Definition \ref{def:intrinsic_alpha_Holder} of H\"older regularity along the H\"ormander fields.
\begin{definition}\label{def:intrinsic_alpha_Holder_local}
Let $u:\Rdd\to \R$. Then, for $a\in]0,1]$, we say that $u\in C^{\a}_{\ZZ_i}(\O)$, $i=1,\dots,d$ if, for any
bounded domain $\O_{0}$ with $\overline{\O}_{0}\subseteq\O$, we have
\begin{equation}
\norm{u}_{C^{\a}_{\ZZ_i}(\O_0)}:=
 \sup_{z\in \O_{0} \atop  0<|\t|<\d_{\O_{0}}} \frac{
 \left|u\left(e^{\t \ZZ_i }(z)\right)-
 u(z)\right|}{|\t|^{\a}} <\infty.
\end{equation}
Moreover, for $\alpha\in]0,\ddY]$, we say that $u\in C^{\a}_{Y}(\Omega)$ if, for any
bounded domain $\O_{0}$ with $\overline{\O}_{0}\subseteq\O$, we have
\begin{equation}
\norm{u}_{C^{\a}_{Y}(\O_0)}:=
 \sup_{z\in\O_{0}\atop 0<|\t|<\d_{\O_{0}}} \frac{
 \left|u\left(e^{\t Y }(z)\right)-
 u(z)\right|}{|\t|^{\frac{\a}{\ddY}}} <\infty.
\end{equation}
\end{definition}
With Definition \ref{def:intrinsic_alpha_Holder_local} at hand, we can now localize the notion of intrinsic H\"older spaces $C^{\a}$ with the following definition, which is completely analogous to Definition \ref{def:C_alpha_spaces} and thus is written in a more compact form for sake of brevity.
\begin{definition}\label{def:C_alpha_spaces_local}
Let $u:\O\to \R$ and $\a>0$. Then $u\in C^{\a}(\O)$ if, for any
bounded domain $\O_{0}$ with $\overline{\O}_{0}\subseteq\O$, the semi-norm defined recursively as
\begin{equation}
  \norm{u}_{C^{\a}(\O_0)}:=
 \begin{cases}
 \norm{u}_{C^{\a}_{Y}(\O_0)}+\sum_{i=1}^{d} \norm{u}_{C^{\a}_{\ZZ_i}(\O_0)}& \quad \text {if } \alpha\in]0,1 \wedge \ddY],\\
\norm{Yu}_{C^{\a-\ddY}(\O_0)}+\sum_{i=1}^{d} \norm{u}_{C^{\a}_{\ZZ_i}(\O_0)}& \quad \text {if }\alpha\in]1 \wedge \ddY, 1 \vee \ddY] \text{ and }\ddY<1, \\
  \norm{u}_{C^{\a}_{Y}(\O_0)}+\sum_{i=1}^{d} \norm{\ZZ_i u}_{C^{\a-1}(\O_0)}& \quad \text {if }\alpha\in]1 \wedge \ddY, 1 \vee \ddY] \text{ and }\ddY >1, \\
  \norm{Y u}_{C^{\a-\ddY}(\O_0)}+\sum_{i=1}^{d} \norm{\ZZ_i u}_{C^{\a-1}(\O_0)}& \quad  \text {if } \alpha>1 \vee \ddY
\end{cases}
\end{equation}
is finite.
\end{definition}
\begin{remark}
It is easy to see that $C^{\a}(\O)$ is a decreasing family in $\a$, meaning that $C^{\a}(\O) \subset C^{\a'}(\O)$ for any $\a > \a'>0 $. 
\end{remark}
We now have the following local version of Theorem \ref{th:main}. Before, for $u\in C^{\a}(\O)$, we define its intrinsic Taylor polynomial centered at $z_0=(t_0,x_0,v_0)$ as
\begin{equation}\label{eq:TAy_pol_non}
 T_{\alpha} u(z_0;z):=\!\!  \sum_{{0\leq \ddY k + (1+\ddY)|\gamma| + |\beta| < \alpha}} \!\! \frac{   Y^k {\partial_{v}^{\beta} \partial_{x}^{\gamma} }  }{k!\, \gamma! \,\beta! } \big( (t_0,x_0,v_0)^{-1} \circ (t,x,v) \big)^{(k,\gamma,\beta)}, \qquad z=(t,x,v)\in\Rdd,
\end{equation}
with $\circ$ as in \eqref{eq:translat_non}, and where we adopted the multi-product notation 
\begin{equation}
(s,h,y)^{(k,\gamma,\beta)} = s^{k} h_1^{\gamma_1} \cdots h_d^{\gamma_d}  y_1^{\beta_1} \cdots y_d^{\beta_d}, \qquad (s,h,y)\in\Rdd.
\end{equation}
Note that the term \emph{polynomial} here is slightly abused, as $T_{\alpha} u(z_0;z)$ is not necessarily a polynomial in the time-increment.

\begin{theorem}\label{th:main_loc}
For any  $\alpha >0$ and for any $u\in C^{\a}(\O)$ the following statements are true:
\begin{description}
 \item [{\sc H\"older spaces characterization.}] For any $k\in\N_0$ and $\gamma,\beta\in\N_0^d$ with ${0\leq \ddY k + (1+\ddY)|\gamma| + |\beta| <\alpha}$, the derivatives 
$ \ Y^k \partial_v^{\beta} \partial_x^{\gamma} u$ exist on $\O$ and 
\begin{equation}\label{eq:maintheorem_part1_loc}
 \| Y^k  \partial_v^{\beta} \partial_x^{\gamma}u \|_{C^{\alpha-\ddY k - (1+\ddY)|\gamma| - |\beta|}(\O_0)}  \leq \| u \|_{C^{\alpha}(\O_0)}
\end{equation}
for any
bounded domain $\O_{0}$ with $\overline{\O}_{0}\subseteq\O$.
 \item [{\sc Taylor formula.}] For any $\z_0\in\Omega$, there exist two bounded domains $U,\Omega_0$ such that $\zeta_0\in \overline{U}\subset \overline{\Omega_0}\subset\O$, and a positive constant $c$, only dependent on $B$, $\alpha$ and $U$, such that 
%
\begin{equation}\label{eq:estim_tay_n_loc}
 \left|u(z)-T_{\alpha} u(z_0;z)\right|\le  c \|u\|_{C^{\a}(\Omega_0)} \|z_0^{-1}\circ z\|^{\a},\qquad z,z_0\in U.
\end{equation}
\end{description}
\end{theorem}
In the homogeneous case, namely when the blocks $B_{11}$, $B_{21}$, $B_{22}$ are null, the proof of Theorem \ref{th:main_loc} is substantially identical to that of Theorem \ref{th:main} with the only additional complexity being that $U$ and $\Omega_0$ have to be chosen in a way that all the integral curves employed to connect $z_0$ and $z$ are contained in $\Omega$. On the other hand, in the general non-homogeneous case, there is a substantial additional difficulty stemming from the fact that the discretization of \eqref{eq:commutator2} with the integral curves of $Z_1, \dots, Z_d$ and $Y$ is more involved than in the homogeneous case. In the rest of the section we give an account of this additional complexity and recall a result from \cite{pagliarani2017intrinsic} that allows to overcome it. 

Let $z=(t,x,v) \in \Rdd$, $w\in\Rd$ and recall the points $z_2$, $z_3$, $z_4$ as defined in \eqref{eq:gamma_k1_bis}, namely
\begin{equation} \label{eq:gamma_k1_bis_non}
 z_2 = e^{\tau Z_w}(t,x,v), \qquad z_3 = e^{\tau^{\ddY }Y}(z_2), \qquad z_4 = e^{- \tau Z_w}(z_3),
\end{equation}
and set 
\begin{equation}
z_5 = e^{-\tau^{\ddY }Y}(z_4), \qquad z_6 = e^{- \tau^{\ddY + 1} Z_{B_{22}w}}(z_5).
\end{equation}
A direct computation shows that
\begin{equation} \label{eq:gamma_k1_ter_non}
z_6 = \big(t, x + \tau^{\ddY  + 1} B_{12} w , v \big) - \t^{2\ddY +1} \bigg(0 , \sum_{n=0}^{\infty} \frac{(-1)^{n}\t^{\ddY n}}{(n+2)!} B^{n+2} (0,w)^\top \bigg), \qquad \t\in\R. 
\end{equation}
Assume that our aim is to move along the $x$ variable only, by an increment $h\in \Rd$. If we adjust the leading order increment in \eqref{eq:gamma_k1_ter_non} by choosing $w = B_{1,2}^{-1} h / |h|$ and $\t = |h|^{\frac{1}{\ddY +1}} $ (recall that $B_{1,2}$ has maximum rank), 
then we are off by an error term of order $\t^{2\ddY +1}$. Note that this error involves both the velocity and position variables as the blocks of $B$ are all non-null. Therefore, if we  make a further correction in the velocity variables by considering 
\begin{align}
\label{eq:def_g1}
g_{w,\t}(z) : = e^{\tau^{2 \ddY + 1} Z_{w'}}(z_6), \qquad \text{with }\quad 
w' = w'(\t,w) = \sum_{n=0}^{\infty} \frac{(-1)^{n}\t^{\ddY n}}{(n+2)!} B^{n+2}_{22} w, 
\end{align}
we fix the velocity variables but the increment in the position variables still differs from $h$ by an error term of order $\d^{2\ddY +1}$. The next lemma allows to connect $z=(t,x,v)$ to $(t,x+h,v)$ moving along the curves above, while keeping $w$ bounded and $|\t|$ controlled in terms of $|h|^{\frac{1}{\ddY +1}}$.  
It concides with \cite[Lemma 3.2]{pagliarani2017intrinsic}, which has been proved for $\ddY =2$. We recall this result without proof, as dealing with a general $\ddY>0$ requires no modification.
\begin{lemma}\label{lem:connect}
There exists $\eps>0$, only dependent on $B$, such that: for any $h\in\Rd 
$ with $|h|\leq \eps$, there exist $w\in\Rd 
$, $\tau\geq 0$ with 
\begin{equation}
|w|=1, \qquad |\t| \leq \frac{2}{\norm{B_{12}}}
 |h|
 ^{\frac{1}{\ddY + 1}},
\end{equation}
such that
\begin{equation}
\label{eq:g1_estim}
g_{w,\t}(z) = (t, x + h, v ), \qquad z=(t,x,v)\in\Rdd,
\end{equation}
where $g_{w,\t}(z)$ is as in \eqref{eq:def_g1}.
\end{lemma}

With Lemma \ref{lem:connect} at hand, which plays the role of identity \eqref{eq:gamma_k1_ter} that we had in the case $Y = \langle   v , \nabla_x   \rangle + \partial_t$, the proof of Theorem \ref{th:main_loc} is essentially the same as in the homogeneous case.


\bibliographystyle{chicago}
\bibliography{Bibtex-Master-3.00}

\end{document}